\documentclass[12pt,a4paper]{amsart}

\usepackage{pgf,tikz}

\newtheorem{theorem}{Theorem}[section]

\newtheorem{lemma}[theorem]{Lemma}
\newtheorem*{lemma*}{Lemma}
\newtheorem{prop}[theorem]{Proposition}

\theoremstyle{remark}

\newtheorem{rmk}[theorem]{Remark}

\newtheorem{exm*}{Example}

\theoremstyle{definition}
\newtheorem{defn}[theorem]{Definition}
\numberwithin{equation}
{section}

\title[Simplicity of $k$-graph $C^*$-algebras]
{Simplicity of $C^*$-algebras associated to row-finite locally convex higher-rank graphs}
\author{David Robertson}
\address{David Robertson, School of Mathematical and Physical Sciences\\
Building V\\
University of Newcastle\\
Callaghan, NSW, 2308\\
AUSTRALIA}
\email{D.Robertson@newcastle.edu.au}
\author{Aidan Sims}
\address{Aidan Sims, School of Mathematics and Applied Statistics\\
Austin Keane Building (15)\\
University of Wollongong\\
NSW, 2522\\
AUSTRALIA}
\email{asims@uow.edu.au}
\date{August 2, 2007}

\newcommand{\Obj}{\operatorname{Obj}}
\newcommand{\Mor}{\operatorname{Mor}}
\newcommand{\id}{\operatorname{id}}
\newcommand{\dom}{\operatorname{dom}}
\newcommand{\cod}{\operatorname{cod}}
\newcommand{\NN}{\mathbb{N}}
\newcommand{\TT}{\mathbb{T}}
\newcommand{\Cc}{\mathcal{C}}
\newcommand{\Mm}{\mathcal{M}}

\newcommand{\tLambda}{\widetilde{\Lambda}}
\newcommand{\clsp}{\overline{\operatorname{span}}}
\newcommand{\MCE}{\operatorname{MCE}}

\begin{document}

\begin{abstract}
In previous work, the authors showed that the $C^*$-algebra
$C^*(\Lambda)$ of a row-finite higher-rank graph $\Lambda$ with no
sources is simple if and only if $\Lambda$ is both cofinal and
aperiodic. In this paper, we generalise this result to row-finite
higher-rank graphs which are locally convex (but may contain
sources). Our main tool is Farthing's ``removing sources"
construction which embeds a row-finite locally convex higher-rank
graph in a row-finite higher-rank graph with no sources in such a way
that the associated $C^*$-algebras are Morita equivalent.
\end{abstract}

\maketitle

\section{Introduction} \label{sec:introduction}

A directed graph is a quadruple $(E^0, E^1, r,s)$: $E^0$ is a
countable set of vertices; $E^1$ is a countable set of directed
edges; and $r,s$ are maps from $E^1$ to $E^0$ which encode the
directions of the edges: an edge $e$ points from the vertex $s(e)$ to
the vertex $r(e)$. In \cite{EW}~and~\cite{KPRR}, $C^*$-algebras were
associated to directed graphs so as to generalise the Cuntz-Krieger
algebras of \cite{CuKr}. These graph $C^*$-algebras have been studied
intensively over the last ten years; see \cite{CBMSbk} for a good
summary of the literature.

For technical reasons related to the groupoid models used to analyse
their $C^*$-algebras, the graphs considered in \cite{KPR, KPRR} were
assumed to be row-finite and to have no sources. This means that
$r^{-1}(v)$ is finite and nonempty for every vertex $v$. To eliminate
the ``no sources" hypothesis, Bates et al. \cite{BaPaRaSz} introduced
a construction known as \emph{adding tails}. Adding tails to a graph
$E$ with sources produces a graph $F$ with no sources so that the
associated $C^*$-algebras $C^*(E)$ and $C^*(F)$ are Morita
equivalent. Proving theorems about $C^*(E)$ then often becomes a
question of identifying hypotheses on $E$ which are equivalent to the
hypotheses of \cite{KPR} for $F$. In \cite{BaPaRaSz}, many important
theorems about $C^*$-algebras of row-finite graphs with no sources
were extended to $C^*$-algebras of arbitrary row-finite graphs using
this strategy.

\medskip

Higher-rank graphs (or $k$-graphs) $\Lambda$ and the associated
$C^*$-algebras $C^*(\Lambda)$ were introduced by Kumjian and Pask in
\cite{KuPa} as a common generalisation of graph $C^*$-algebras and
the higher-rank Cuntz-Krieger algebras developed by G. Robertson and
Steger in \cite{RoSt, RobSt}. As had \cite{KPR, KPRR} for graph
$C^*$-algebras, \cite{KuPa} studied $k$-graph $C^*$-algebras using a
groupoid model, and technical considerations associated to this model
made it necessary to restrict attention to $k$-graphs which were
row-finite and had no sources. Like graph $C^*$-algebras, $k$-graph
$C^*$-algebras have received substantial attention in recent years.
Unlike graph $C^*$-algebras, however, some fundamental
structure-theoretic questions regarding $k$-graph algebras have not
yet been answered, primarily due to the combinatorial complexities of
$k$-graphs themselves.

In \cite{RoSi}, the authors established that the \emph{cofinality}
and \emph{aperiodicity} conditions formulated in \cite{KuPa} as
sufficient conditions for simplicity of the $C^*$-algebra of a
row-finite higher-rank graph with no sources are in fact also
necessary. The resulting simplicity theorem is an exact analogue of
the original simplicity theorem for $C^*$-algebras of row-finite
graphs with no sources \cite{KPR}.

In this paper, we extend our previous simplicity result to a large
class of higher-rank graphs with sources. To do this we use
Farthing's \emph{removing sources} construction which produces from a
higher-rank graph $\Lambda$ a higher-rank graph $\overline{\Lambda}$
with no sources such that if $\Lambda$ is row-finite, then
$C^*(\Lambda)$ and $C^*(\overline{\Lambda})$ are Morita equivalent
\cite{Fa}. By establishing how infinite paths in $\overline{\Lambda}$
are related to boundary paths in $\Lambda$, we use Farthing's results
to generalise the simplicity result of \cite{RoSi} to the
\emph{locally convex} row-finite $k$-graphs with sources considered
in \cite{RaSiYe}. We should point out that the restriction to locally
convex $k$-graphs is not forced on us by Farthing's results, which
are applicable for arbitrary row-finite graphs. Indeed, local
convexity plays no role in Farthing's analysis, and initially we had
no expectation that it would impinge upon our analysis here. However,
it turns out, interestingly enough, that local convexity is needed to
ensure that the natural projection of $\overline{\Lambda}$ onto
$\Lambda$ extends to a projection from the space of infinite paths of
$\overline{\Lambda}$ to the space of boundary paths of $\Lambda$. We
use this projection to translate aperiodicity and cofinality
hypotheses on $\overline{\Lambda}$ to analogous conditions on
$\Lambda$.

\vskip0.5em plus 0em minus 0.1em

We begin the paper with preliminary notation and definitions in
Section~\ref{sec:preliminaries}. We also outline Farthing's removing
sources construction in the setting of row-finite locally convex
higher-rank graphs, and explore the relationship between boundary
paths in $\Lambda$ and infinite paths in $\overline{\Lambda}$. In
Section~\ref{sec:simplicity}, we turn to our main objective. Using
the results of the previous section, we formulate notions of
cofinality and aperiodicity for a locally convex row-finite
higher-rank graph $\Lambda$ which are equivalent to the corresponding
conditions of \cite{KuPa} for $\overline{\Lambda}$. Combining this
with Farthing's Morita equivalence between $C^*(\Lambda)$ and
$C^*(\overline{\Lambda})$ and the results of \cite{RoSi} yields the
desired simplicity theorem. In Section~\ref{sec:idealstructure}, we
apply the same methods to a number of other results of \cite{RoSi}
concerning the relationship between graph-theoretic properties of a
$k$-graph $\Lambda$ and the ideal-structure of $C^*(\Lambda)$.

\section{Preliminaries} \label{sec:preliminaries}

In this section we summarise the standard notation and conventions
for higher-rank graphs.

We regard $\NN=\{0,1,2,\dots\}$ as a semigroup under addition. We
write $\NN^k$ for the set of $k$-tuples $n=(n_1,n_2,\dots,n_k)$,
$n_i\in\NN$, which we regard as a semigroup under pointwise addition
with identity $0 = (0,0,\dots,0)$. We denote the canonical generators
of $\NN^k$ by $e_1,e_2,\dots,e_k$. Given $m,n\in\NN^k$, we say $m\leq
n$ if $m_i\leq n_i$ for $1\leq i\leq k$. We write $m\vee n$ for the
coordinate-wise maximum of $m$ and $n$, and $m\wedge n$ for the
coordinate-wise minimum. Unless otherwise indicated through
parentheses, $\vee$ and $\wedge$ always take precedence over addition
and subtraction, so for example $m\vee n-n=(m\vee n)-n$.

Note that $\vee$ and $\wedge$ distribute over addition and subtraction:
\begin{equation} \label{minimumdistributesoveraddition}
 a\wedge b + c = (a+c)\wedge(b+c) \ \mbox{ and } \ a\vee b+c = (a+c)\vee(b+c)
\end{equation}
for all $a,b,c\in\NN^k$, and
\begin{equation} \label{minimumdistributesoversubtraction}
 a\wedge b - c = (a-c)\wedge (b-c) \ \mbox{ and } \ a\vee b-c = (a-c)\vee(b-c)
\end{equation}
(as elements of $\mathbb{Z}^k$) for all $a,b,c\in\NN^k$.

\subsection{Higher-rank graphs} \label{subsec:higher-rankgraphs}
The notion of a higher-rank graph is best phrased in terms of
categories. For the basics of categories we refer the reader to
Chapter~1 of \cite{Ma}. We assume that all our categories are small
in the sense that $\Obj(\mathcal{C})$ and $\Mor(\mathcal{C})$ are
sets. Given a category $\mathcal{C}$, we identify $\Obj(\mathcal{C})$
with $\{\id_o:o\in\Obj(\mathcal{C})\}\subset\Mor(\mathcal{C})$, and
we write $c\in\mathcal{C}$ to mean $c\in\Mor(\mathcal{C})$. We write
composition in our categories as juxtaposition; that is, for
$c_1,c_2\in\mathcal{C}$ with $\dom(c_1)=\cod(c_2)$, $c_1c_2$ means
$c_1\circ c_2$. When convenient, we regard $\NN^k$ as a category with
just one object, and composition implemented by addition. We write
$\Cc\times_{\Obj(\Cc)} \Cc$ for the set
$\{(c_1,c_2):\dom(c_1)=\cod(c_2)\}$ of composable pairs in $\Cc$.

Fix $k\in\NN\setminus\{0\}$. A \emph{graph of rank $k$} or
\emph{$k$-graph} is a countable category $\Lambda$ equipped with a
functor $d:\Lambda\to\NN^k$ called the \emph{degree functor}, which
satisfies the \emph{factorisation property}: Given
$\lambda\in\Lambda$ with $d(\lambda)=m+n$, there are unique paths
$\mu,\nu\in\Lambda$ such that $d(\mu)=m$, $d(\nu)=n$ and
$\lambda=\mu\nu$. \emph{Higher-rank graph} is the generic term when
the rank $k$ is not specified.

Given a $k$-graph $\Lambda$ one can check using the factorisation
property that $d^{-1}(0) = \{\id_v : v \in \Obj(\Lambda)\}$. We
define $r,s : \Lambda \to d^{-1}(0)$ by $r(\lambda) :=
\id_{\cod(\lambda)}$ and $s(\lambda) := \id_{\dom(\lambda)}$, so
$\lambda = r(\lambda)\lambda = \lambda s(\lambda)$ for all $\lambda
\in \Lambda$. We think of $d^{-1}(0)$ as the vertices of $\Lambda$,
and refer to $r(\lambda)$ as the \emph{range} of $\lambda$ and to
$s(\lambda)$ as the \emph{source} of $\lambda$.

For $n\in\NN^k$ we write $\Lambda^n:=d^{-1}(n)$, so $\Lambda^0$
is the collection of vertices. For $v\in\Lambda^0$ and
$E\subset\Lambda$, we write $v E$ for $\{\lambda\in
E:r(\lambda)=v\}$, and $Ev:=\{\lambda\in E:s(\lambda)=v\}$.

For $\lambda\in\Lambda$ with $d(\lambda)=l$, and $0\leq m\leq
n\leq l$, the factorisation property ensures that there are
unique paths $\lambda'\in\Lambda^{m}$,
$\lambda''\in\Lambda^{n-m}$ and $\lambda'''\in\Lambda^{l-n}$
such that $\lambda=\lambda'\lambda''\lambda'''$. We write
$\lambda(0,m),\lambda(m,n)$ and $\lambda(n,l)$ for
$\lambda,\lambda'$ and $\lambda'''$ respectively.

We say a $k$-graph $\Lambda$ is \emph{row finite} if
$|v\Lambda^n|<\infty$ for all $v\in\Lambda^0$ and $n\in\NN^k$, and
that has \emph{no sources} if $v\Lambda^n\neq\emptyset$ for all
$v\in\Lambda^0$ and $n\in\NN^k$. We say $\Lambda$ is \emph{locally
convex} if for distinct $i,j \in \{1, \dots, k$, and paths
$\lambda\in\Lambda^{e_i}$ and $\mu\in\Lambda^{e_j}$ such that
$r(\lambda)=r(\mu)$, the sets $s(\lambda)\Lambda^{e_j}$ and
$s(\mu)\Lambda^{e_i}$ are non-empty.

For $n\in\NN^k$, we write
\[
\Lambda^{\leq n} := \{\lambda\in\Lambda:d(\lambda)\leq n
                    \text{ and } s(\lambda)\Lambda^{e_i}=\emptyset
                    \text{ whenever } d(\lambda)+e_i\leq n \}
\]
When $\Lambda$ has no sources, $\Lambda^{\leq n}=\Lambda^n$. For a
locally convex $k$-graph $\Lambda$,  $v\Lambda^{\leq n}\neq\emptyset$
for all $n\in\NN^k$ and $v\in\Lambda^0$, but $v\Lambda^n \not=
\emptyset$ for all $n,v$ if and only if $\Lambda$ has no sources.

\subsection{Boundary Paths and Infinite Paths}
Fix $k>0$ and $m\in(\NN\cup\{\infty\})^k$. Then $\Omega_{k,m}$
denotes the category with objects
$\Obj(\Omega_{k,m})=\{p\in\NN^k:p\leq m\}$, morphisms
$\Mor(\Omega_{k,m})=\{(p,q):p,q\in\NN^k,p\leq q\leq m\}$ and
$\dom(p,q)=q$, $\cod(p,q)=p$, $\id(p)=(p,p)$,
$(p,q)\circ(q,t)=(p,t)$. The formula $d(p,q)=q-p$ defines a functor
$d:\Omega_{k,m}\to \NN^k$, and the pair $(\Omega_{k,m},d)$ is a
row-finite locally convex $k$-graph.

Let $\Lambda$ be a row-finite locally convex $k$-graph. A
degree-preserving functor $x:\Omega_{k,m}\to\Lambda$ is a
\emph{boundary path of $\Lambda$} if
\[
(p\leq m \mbox{ and } p_i=m_i) \Longrightarrow
x(p)\Lambda^{e_i}=\emptyset \quad \mbox{ for all } p\in\NN^k, 1\leq i\leq k.
\]
We regard of $m$ as the degree of $x$ and denote it $d(x)$, and we
regard $x(0)$ as the range of $x$ and denote it $r(x)$. If
$m=(\infty,\dots,\infty)$ we call $x$ an \emph{infinite path}. We
write $\Lambda^{\leq\infty}$ for the collection of all boundary paths
and $\Lambda^{\infty}$ for the collection of all infinite paths of
$\Lambda$.

For $x\in\Lambda^{\leq\infty}$ and $n\in\NN^k$ with $n\leq d(x)$,
there is a boundary path $\sigma^n(x):\Omega_{k,d(x)-n}\to\Lambda$
defined by $\sigma^n(x)(p,q):=x(p+n,q+n)$ for all $p \le q \le
d(x)-n$. Given $x\in v\Lambda^{\leq\infty}$ and $\lambda\in\Lambda$
with $s(\lambda)=v$ there is a unique boundary path $\lambda
x:\Omega_{k,m+d(\lambda)}\to\Lambda$ satisfying $(\lambda
x)(0,d(\lambda)) = \lambda$ and $(\lambda x)(d(\lambda),d(\lambda)+p)
= x(0,p)$ for all $p\leq d(x)$. For $x\in\Lambda^{\leq\infty},
\lambda\in\Lambda r(x)$ and $p\leq d(x)$, we have $x(0,p)\sigma^p(x)
= x = \sigma^{d(\lambda)}(\lambda x)$. The set $\Lambda^{\le \infty}$
has similar properties to the $\Lambda^{\le n}$: if $\Lambda$ has no
sources, then $\Lambda^{\le \infty} = \Lambda^\infty$; and for
$\Lambda$ locally convex, $v\Lambda^{\leq\infty}$ is non-empty for
all $v\in\Lambda^0$, but $v\Lambda^\infty \not=\emptyset$ for all $v$
if and only if $\Lambda$ has no sources.

\subsection{$C^*$-algebras associated to $k$-graphs}
Let $\Lambda$ be a row-finite locally convex $k$-graph. Let $A$ be a
$C^*$-algebra, and let $\{t_\lambda:\lambda\in\Lambda\}$ be a
collection of partial isometries in $A$. We call $\{t_\lambda :
\lambda \in \Lambda\}$ a \emph{Cuntz-Krieger $\Lambda$-family} if
\begin{list}{}{\setlength{\itemindent}{.6em}\setlength{\labelwidth}{2em}}
\item[(i)] $\{t_v:v\in\Lambda^0\}$ is a collection of mutually
    orthogonal projections;
\item[(ii)] $t_\lambda t_\mu = t_{\lambda\mu}$ whenever
    $s(\lambda)=r(\mu)$;
\item[(iii)] $t_\mu^*t_\mu=t_{s(\mu)}$ for all $\mu\in\Lambda$; and
\item[(iv)] $t_v=\sum_{\lambda\in v\Lambda^{\leq n}}t_\lambda
    t_\lambda^*$ for all $v\in\Lambda^0, n\in\NN^k$.
\end{list}

There is a $C^*$-algebra $C^*(\Lambda)$ generated by a Cuntz-Krieger
$\Lambda$-family $\{s_\lambda : \lambda \in \Lambda\}$ which is
universal in the following sense: if
$\{t_\lambda:\lambda\in\Lambda\}$ is another Cuntz-Krieger
$\Lambda$-family in a $C^*$-algebra $A$, then there is a homomorphism
$\pi_t:C^*(\Lambda)\to A$ satisfying $\pi_t(s_\lambda)=t_\lambda$ for
all $\lambda\in\Lambda$.

\subsection{Removing sources from $k$-graphs}

The following construction is due to Farthing \cite{Fa}, and we refer
the reader there for details and proofs. We have modified the
formulation given in~\cite{Fa} slightly to streamline later proofs;
lengthy calculations show that the resulting $\tLambda$ is isomorphic
to Farthing's $\overline{\Lambda}$.

We present Farthing's construction only for row-finite locally convex
$k$-graphs because that is the generality in which we will be working
for the remainder of the paper. However, the construction makes sense
for arbitrary $k$-graphs as long as the appropriate notion of a
boundary path is used.

Fix a row-finite locally convex $k$-graph $\Lambda$. Define
\[
V_\Lambda := \{(x;m) : x \in \Lambda^{\le\infty}, m \in \NN^k\}
\]
and
\[
P_\Lambda := \{(x; (m,n)) : x \in \Lambda^{\le\infty}, m \le n \in \NN^k\}.
\]

The relation $\sim$ on $V_\Lambda$ defined by $(x;m) \sim
(y;n)$ if and only if
\begin{itemize}
\item[(V1)] $x(m \wedge d(x)) = y(n \wedge d(y))$; and
\item[(V2)] $m - m \wedge d(x) = n - n \wedge d(y)$
\end{itemize}
is an equivalence relation. We denote the equivalence class of
$(x;m)$ under $\sim$ by $[x;m]$. The relation $\approx$ on
$P_\Lambda$ defined by $(x;(m,n)) \approx (y; (p,q))$ if and only if
\begin{itemize}
\item[(P1)] $x(m \wedge d(x), n \wedge d(x)) = y(p \wedge
    d(y), q \wedge d(y))$;
\item[(P2)] $m - m \wedge d(x) = p - p \wedge d(y)$; and
\item[(P3)] $n - m = q - p$
\end{itemize}
is also an equivalence relation. We denote the equivalence class of
$(x;(m,n))$ under $\approx$ by $[x;(m,n)]$.

Theorem~3.24 of \cite{Fa} implies that there is a row-finite
$k$-graph $\tLambda$ with objects $V_\Lambda/\sim$, morphisms
$P_\Lambda/\approx$, and structure-maps specified by the following
formulae:
\begin{align*}
\widetilde{r}([x;(m,n)]) &:= [x;m], \\
\widetilde{s}([x;(m,n)]) &:= [x;n], \\
\widetilde{\id}([x;m]) &:= [x;(m,m)], \\
[x;(m,n)] \widetilde{\circ} [y;(p,q)] &:= [x(0, n \wedge
d(x))\sigma^{p \wedge d(y)}(y); (m, n+q-p)], \quad\text{and}\\
\widetilde{d}([x;(m,n)]) &:= n-m.
\end{align*}
Theorem~3.24 further states that $(\widetilde{\Lambda},
\widetilde{d})$ is a row-finite $k$-graph with no sources.

For $\lambda \in \Lambda$ and boundary paths $x,y \in
s(\lambda)\Lambda^{\le \infty}$, the elements $(\lambda x; (0,
d(\lambda)))$ and $(\lambda y; (0, d(\lambda)))$ of $P_\Lambda$ are
equivalent under ${\sim}$. Hence there is a map $\iota : \Lambda \to
\widetilde{\Lambda}$ satisfying $\iota(\lambda) := [\lambda x; (0,
d(\lambda))]$ for any $x \in s(\lambda) \Lambda^{\le \infty}$. Indeed
$\iota$ is an injective $k$-graph morphism, and hence an isomorphism
of $\Lambda$ onto $\iota(\Lambda) \subset \widetilde{\Lambda}$.

Theorem 3.29 of \cite{Fa} shows that $\sum_{v \in
\Lambda^0} s_{\iota(v)}$ converges to a full projection $P$ in
the multiplier algebra of $C^*(\widetilde{\Lambda})$, and that
$P C^*(\widetilde{\Lambda}) P = C^*(\{s_{\iota(\lambda)} :
\lambda \in \Lambda\}) \cong C^*(\Lambda)$. In particular,
$C^*(\widetilde{\Lambda})$ is Morita equivalent to
$C^*(\Lambda)$.

In light of the preceding two paragraphs, we can --- and do ---
regard $\Lambda$ as a subset of $\widetilde{\Lambda}$, dropping the
inclusion map $\iota$, and regard $C^*(\Lambda)$ as a
$C^*$-subalgebra of $C^*(\widetilde{\Lambda})$ also.

\subsection{From boundary paths to infinite paths and back}

Our aim in this paper is to combine Farthing's results with the
results of \cite{RoSi} to characterise simplicity of
$C^*(\Lambda)$ in terms of the structure of
$\Lambda^{\leq\infty}$. Applied directly, the results of
\cite{RoSi} characterise the simplicity of $C^*(\Lambda)$ in
terms of the structure of $\tLambda^\infty$. Hence we begin by
establishing how elements of $\Lambda^{\leq\infty}$ correspond
to elements of $\tLambda^\infty$.

For an element $x\in\Lambda^{\leq\infty}$ and $n\in\NN^k$,
define
\[
 [x;(n,\infty)]:\Omega_k\to\tLambda \mbox{ by } [x;(n,\infty)](p,q):=[x;(n+p,n+q)].
\]
Our notation was chosen to suggest the relationship between
$[x;(n,\infty)]$ and $[x;(n,p)]$, $p \ge n$; however, the reader
should note that $\approx$ is not defined on boundary paths, and in
particular $[x;(n,\infty)]$ is not itself an equivalence class under
$\approx$.

Define $\pi:\tLambda\to\Lambda$ by $\pi([x;(m,n)]) = [x;(m\wedge
d(x),n\wedge d(x)]$ for all $x\in\Lambda^{\leq\infty}$ and $m\leq n
\in\NN^k$. Note that $\pi([x;(m,n)]) \in \iota(\Lambda)\subseteq
\tLambda$, so we regard $\pi$ as a $k$-graph morphism from $\tLambda$
to $\Lambda$, and identify $\pi([x;(m,n)])$ with $x(m\wedge
d(x),n\wedge d(x))$. That $\pi$ is well defined follows immediately
from (P1). It is straightforward to check that $\pi$ is a functor, is
surjective onto $\Lambda$ and is a projection in the sense that $\pi
\circ \pi = \pi$.

\begin{lemma} \label{lemma:pathnotation}
Let $\Lambda$ be a locally convex row-finite $k$-graph. For
$x\in\Lambda^{\leq\infty}$ and $m\leq n\in\NN^k$ we have
\[
 [x;(m,n)]=[\sigma^{m\wedge d(x)}(x);(m-m\wedge d(x),n-m\wedge d(x))].
\]
Also $(m-m\wedge d(x))\wedge d(\sigma^{m\wedge d(x)}(x))=0$. In
particular, for each $\lambda \in \tLambda$ there exist $y \in
\Lambda^{\leq\infty}$ and $p \in \NN^k$ such that $\lambda =
[y; (p, p+d(\lambda))]$ and $p \wedge d(y) = 0$.
\end{lemma}

\begin{proof}
We begin by verifying $(m-m\wedge d(x))\wedge d(\sigma^{m\wedge
d(x)}(x))=0$. For each $i\in\{1,\dots,k\}$, either $(m\wedge
d(x))_i=m_i$ or $(m\wedge d(x))_i=d(x)_i$. Hence, either
$(m-m\wedge d(x))_i=0$ or $(d(x)-m\wedge d(x))_i=0$ or both.
Since $d(\sigma^{m\wedge d(x)}(x))=d(x)-m\wedge d(x)$, we
therefore obtain
\begin{equation} \label{eqn:minequalzero}
 (m-m\wedge d(x))\wedge d(\sigma^{m\wedge d(x)}(x))=0
\end{equation}
as required.

To see that $[x;(m,n)]=[\sigma^{m\wedge d(x)}(x);(m-m\wedge
d(x),n-m\wedge d(x))]$, we must check that $(x;(m,n)) \approx
(\sigma^{m\wedge d(x)}(x);(m-m\wedge d(x),n-m\wedge d(x)))$.

For (P1), we note that
\begin{equation} \label{star}
  x(m\wedge d(x),n\wedge d(x))= \sigma^{m\wedge d(x)}(x)(0,n\wedge d(x)-m\wedge d(x))
\end{equation}

By~\eqref{eqn:minequalzero} we may replace the $0$ in the right hand
side of~\eqref{star} with $(m-m\wedge d(x))\wedge d(\sigma^{m\wedge
d(x)}(x))$, giving
\[ \label{star2}
x(m\wedge d(x),n\wedge d(x))
 = \sigma^{m\wedge d(x)}(x)((m-m\wedge d(x))\wedge d(\sigma^{m\wedge d(x)}(x)),
                            n\wedge d(x)-m\wedge d(x)))
\]
Taking $a=n,b=d(x)$ and $c=m\wedge d(x)$
in~\eqref{minimumdistributesoversubtraction} we have
\begin{eqnarray*}
 n\wedge d(x)-m\wedge d(x) &=& (n-m\wedge d(x))\wedge (d(x)-m\wedge d(x)) \\
 &=& (n-m\wedge d(x))\wedge d(\sigma^{m\wedge d(x)}(x)).
\end{eqnarray*}
Hence~\eqref{star2} implies that
\[\begin{split}
 x(m\wedge {}&d(x),n\wedge d(x)) \\
  &= \sigma^{m\wedge d(x)}(x)((m-m\wedge d(x))\wedge d(\sigma^{m\wedge d(x)}(x)),
                              (n-m\wedge d(x))\wedge d(\sigma^{m\wedge d(x)}(x))
\end{split}\]
establishing (P1). For (P2), we calculate:
\begin{align*}
 m-m\wedge d(x) &= (m-m\wedge d(x))-0 \\
 &= (m-m\wedge d(x))-(m-m\wedge d(x))\wedge d(\sigma^{m\wedge d(x)}(x)).
\end{align*}
For (P3), just note that $n-m = (n-m\wedge d(x))-(m-m\wedge d(x))$.

For the last statement of the lemma, write $\lambda=[x;(m,n)]$ for
some $x\in\Lambda^{\leq\infty}$ and $m,n\in\NN^k$, and take
$y=\sigma^{m\wedge d(x)}(x)$ and $p=m-m\wedge d(x)$.
\end{proof}

\begin{lemma} \label{lemmalocalperiodicity2}
Let $\Lambda$ be a row-finite locally convex $k$-graph and let
$v\in\Lambda^0$. Suppose $x\in v\Lambda^{\leq\infty}$ and
$p\in\NN^k$ satisfy $p\wedge d(x)=0$. Then for any other $z\in
v\Lambda^{\leq\infty}$ we have $p\wedge d(z)=0$ and
$[x;(0,p)]=[z;(0,p)]$.
\end{lemma}
\begin{proof}
Fix $z\in v\Lambda^{\leq\infty}$. We must show that $p_i\neq 0$
implies $d(z)_i=0$. Suppose $i\in\{1,\dots,k\}$ satisfies $p_i\neq
0$. Then $p\wedge d(x)_i=0$ implies that $d(x)_i=0$. Since $x$ is a
boundary path it follows that $v\Lambda^{e_i}=\emptyset$, so
$d(z)_i=0$ also. One now easily verifies (P1)--(P3) directly to see
that $[x;(0,p)]=[z;(0,p)]$.
\end{proof}

Before stating the next proposition, we need some notation. Fix
$\mu,\nu\in\Lambda$. We say that $\lambda\in\Lambda$ is a
\emph{common extension} of $\mu$ and $\nu$ if $\lambda = \mu\mu'=
\nu\nu'$ for some $\mu',\nu' \in \Lambda$. If $\lambda$ is a common
extension of $\mu$ and $\nu$, then in particular, $d(\lambda) \ge
d(\mu) \vee d(\nu)$. We say that a common extension $\lambda$ of
$\mu$ and $\nu$ is a \emph{minimal} common extension if
$d(\lambda)=d(\mu)\vee d(\nu)$. We write $\MCE(\mu,\nu)$ for the
collection of all minimal common extensions of $\mu$ and $\nu$.

Now fix $v \in \Lambda^0$ and let $E\subset v\Lambda$. We say that
$E$ is \emph{exhaustive} if for each $\lambda\in v\Lambda$ there
exists $\mu\in E$ such that $\MCE(\lambda,\mu)\neq \emptyset$. If $E$
is also finite, we call $E$ \emph{finite exhaustive}.

\begin{prop} \label{prop:infinitepaths}
Let $(\Lambda,d)$ be a row-finite, locally convex $k$-graph. Then for
each infinite path $y\in\tLambda^\infty$, there exist a unique $p_y
\in \NN^k$ and $\pi(y) \in\Lambda^{\leq\infty}$ such that $p_y \wedge
d(\pi(y)) = 0$ and $y = [\pi(y); (p_y, \infty)]$. We then have
$\pi(y(m,n)) = \pi(y)(m \wedge d(\pi(y)), n \wedge d(\pi(y)))$ for
all $m,n \in \NN^k$.
\end{prop}
\begin{proof}
We first argue existence. Fix $y\in\tLambda^\infty$. For each
$n\in\NN^k$ there exists an $x_n\in\Lambda^{\leq\infty}$ and
$p(n)\in\NN^k$ such that $[x_n;(p(n),p(n)+n)]=y(0,n)$. By
Lemma~\ref{lemma:pathnotation} we may assume that $p(n)\wedge
d(x_n)=0$ for all $n\in\NN^k$. This forces $p(n)=p(0)$ and
$r(x_n)=r(x_0)$ for all $n\in\NN^k$. We will henceforth just
write $p$ for $p(0)$. For $a,b \in \NN^k$ with $a \le b$, we
have
\[
 [x_a;(p,p+a)]=y(0,a)=[x_b;(p,p+a)].
\]
Since $p \wedge d(x_a) = 0 = p \wedge d(x_b)$, we therefore
have
\begin{equation} \label{eqn:equalpaths}
 x_a(0,a \wedge d(x_a))=x_b(0,a\wedge d(x_a)) \quad\text{for $a \le b \in \NN^k$.}
\end{equation}
Define $m \in (\NN \cup \{\infty\})^k$ by
\[\textstyle
 m:=\bigvee_{n\in\NN^k} n \wedge d(x_n).
\]
Fix $q \in \NN^k$ with $q \le m$. Suppose that $a,b \in \NN^k$
satisfy $a \wedge d(x_a), b \wedge d(x_b) \ge n$. Since $q \le
a \wedge d(x_a)$, we may use Equation~\eqref{eqn:equalpaths} to
calculate:
\[
x_a(0,q)
= (x_a(0, a \wedge d(x_a)))(0,q)
= (x_{a \vee b}(0, a \wedge d(x_a)))(0,q)
= x_{a \vee b}(0,q).
\]
Similarly, $x_b(0,q) = x_{a \vee b}(0,q)$, and in particular,
$x_a(0,q) = x_b(0,q)$.
%
Hence there is a unique graph morphism $x : \Omega_{k,m} \to
\Lambda$ such that for $q,q' \in \NN^k$ with $q \le q' \le m$,
\[
x(q,q') := x_a(q,q')\text{ for any $a\in\NN^k$ such that
$a \wedge d(x_a) \geq q'$.}
\]

We claim that $x\in\Lambda^{\leq\infty}$. Suppose $q\in\NN^k$ and $1
\le i \le k$ satisfy $q\leq m$ and $q_i = m_i$. We must show that
$x(q)\Lambda^{e_i} = \emptyset$. We first claim that $q \wedge d(x_q)
= q$. To see this, we argue by contradiction. Suppose that $d(x_q)_j
< q_j$ for some $j$. Since $x_q$ is a boundary path, $x_q(q \wedge
d(x_q))\Lambda^{e_j} = \emptyset$. By~\eqref{eqn:equalpaths},
$x_{q'}(q \wedge d(x_q)) = x_q(q \wedge d(x_q))$ for all $q' \ge q$,
and this forces $d(x_{q'})_j = d(x_q)_j < q_j$ for all $q' \ge q$. In
particular, since $q \le m$, this contradicts the definition of $m$.
This proves the claim. Now, by definition of $m$, we also have $(q +
e_i) \wedge d(x_{q + e_i}) = q_i$, and hence $x_{q +
e_i}(q)\Lambda^{e_i} = \emptyset$. By~\eqref{eqn:equalpaths}, we have
$x_{q + e_i}(q) = x_{q}(q) = x(q)$, so $x(q)\Lambda^{e_i} =
\emptyset$ as claimed.

Now we aim to show that $y=[x;(p,\infty)]$. To do this we show that
\[
 (x;(p,p+n)) \approx (x_n;(p,p+n))
\]
for all $n\in\NN^k$. Fix $n\in\NN^k$.

Condition (P1) follows directly from the definition of $x$.

For (P2), note that $p\wedge d(x_n)=0$ for all $n$ implies $p\wedge
d(x)=0$. So
\[
 p-p\wedge d(x) = p = p-p\wedge d(x_n).
\]
Condition (P3) is immediate. Hence $y=[x;(p,\infty)]$. Since $p
\wedge d(x) = 0$ as observed above, taking $p_y := p$ and $\pi(y) :=
x$ establishes existence.

For uniqueness, suppose that $x'$ and $p'$ have the same properties.
Then $[x';(p', p' + m)] = y(0) = [\pi(y);(p_y, p_y + m)]$ for all $m
\in \NN^k$. Since $p' \wedge d(x') = 0 = p_y \wedge d(\pi(y))$,
condition~(P1) forces $x'(0, m \wedge d(x')) = \pi(y)(0, m \wedge
d(\pi(y)))$ for all $m \in \NN^k$. In particular, $m \wedge d(x') = m
\wedge d(\pi(y))$ for all $m \in \NN^k$, so $d(x') = d(\pi(y))$, and
since $x'$ and $\pi(y)$ agree on all initial segments, they must be
equal. Condition~(P2) together with $p' \wedge d(x') = p_y \wedge
d(\pi(y)) = 0$ forces $p' = p$. Thus $p_y$ and $\pi(y)$ are unique.

For the final statement, fix $m,n \in \NN^k$. By the preceding
paragraphs, $y(m,n) = [\pi(y); (p_y + m, p_y + n)]$, so $\pi(y(m,n))
= \pi(y)((p_y + m) \wedge d(\pi(y)), (p_y + n) \wedge d(\pi(y)))$.
Since $p_y \wedge d(\pi(y)) = 0$, we have $(p_y + m) \wedge d(\pi(y))
= m \wedge d(\pi(y))$ and similarly for $n$, and this completes the
proof.
\end{proof}

\section{Simplicity} \label{sec:simplicity}

\begin{defn}
Let $(\Lambda,d)$ be a row-finite locally convex $k$-graph. We say
$\Lambda$ is \emph{cofinal} if, for each $x\in\Lambda^{\leq\infty}$
and $v\in\Lambda^0$, there exists $n\in\NN^k$ such that $n\leq d(x)$
and $v\Lambda x(n)$ is nonempty.
\end{defn}

\begin{defn}\label{dfn:nlp}
Let $(\Lambda,d)$ be a row-finite $k$-graph, fix $v\in\Lambda^0$, and
fix $m\neq n\in\NN^k$. We say $\Lambda$ has \emph{local periodicity
$m,n$ at $v$} if for every $x\in v\Lambda^{\leq\infty}$, we have $m -
m \wedge d(x) = n - n \wedge d(x)$ and $\sigma^{m \wedge d(x)}(x) =
\sigma^{n \wedge d(x)}(x)$. We say $\Lambda$ has \emph{no local
periodicity} if $\Lambda$ does not have local periodicity $m,n$ at
$v\in\Lambda^0$ for any $m\neq n\in\NN^k$ and $v\in\Lambda^0$.
\end{defn}

\begin{rmk}\label{rmk:nlp reduces}
Note that if $\Lambda$ has no sources, then every boundary path is an
infinite path. In particular, $m \wedge d(x) = m$ and $n \wedge d(x)
= n$, for all $m,n \in \NN^k$ and all $x \in \Lambda^{\le \infty}$.
Hence the definitions of cofinality, local periodicity and no local
periodicity presented above reduce to the definitions of the same
conditions given in \cite{RoSi} when $\Lambda$ has no sources.
\end{rmk}

\begin{theorem} \label{theorem:simplicity}
Let $(\Lambda,d)$ be a row-finite locally convex $k$-graph. Then
$C^*(\Lambda)$ is simple if and only if both of the following
conditions hold:
\begin{itemize}
 \item[(1)] $\Lambda$ is cofinal; and
 \item[(2)] $\Lambda$ has no local periodicity.
\end{itemize}
\end{theorem}

We prove Theorem \ref{theorem:simplicity} at the end of the
section. To do so, we first establish two key Propositions.

\begin{prop} \label{prop:cofinal}
Let $(\Lambda,d)$ be a row-finite, locally convex $k$-graph.
Then $\Lambda$ is cofinal if and only if $\tLambda$ is cofinal.
\end{prop}
\begin{proof}
($\Longrightarrow)$ Suppose $\Lambda$ is cofinal. Fix
$z\in\tLambda^{\infty}$ and $[y;p]\in\tLambda^0$. Proposition
\ref{prop:infinitepaths} implies $z=[x;(m,\infty)]$ for some
$x\in\Lambda^{\leq\infty}$ and $m\in\NN^k$. We must show that there
exists $p' \in \NN^k$ such that $[y;p] \tLambda [x;(m,\infty)](p')
\not=\emptyset$.

 Fix $q\geq p$ such that
$y(q\wedge d(y))\Lambda^{e_i}=\emptyset$ whenever $d(y)_i<\infty$.
Since $\Lambda$ is cofinal,
\[
 y(q\wedge d(y))\Lambda x(m\wedge d(x)+n)\neq\emptyset
 \quad\text{for some $n\leq d(x)-m\wedge d(x)$.}
\]
Fix a path $\lambda\in y(q\wedge d(y))\Lambda x(m\wedge d(x)+n)$.
Lemma~2.10 of \cite{RaSiYe} implies that $y':=y(0,q\wedge
d(y))\,\lambda \,\sigma^{m\wedge d(x)+n}(x)\in\Lambda^{\leq\infty}$.
We claim that $q\wedge d(y)=q\wedge d(y')$. We have
\begin{eqnarray*}
 q\wedge d(y') &=& q\wedge d(y(0,q\wedge d(y))\,\lambda \,\sigma^{m\wedge
d(x)+n}(x)) \\
 &=& q\wedge(q\wedge d(y)+d(\lambda)+(d(x)-m\wedge d(x)-n)).
\end{eqnarray*}
For $i\in\{1,\dots,k\}$ such that $d(y)_i<\infty$, we have $y(q\wedge
d(y))\Lambda^{e_i}=\emptyset$ by choice of $q$, so $d(\lambda
\,x(m\wedge d(x)+n,d(x)))_i=0$ and
\[
 (q\wedge(q\wedge d(y)+d(\lambda)+(d(x)-m\wedge d(x)-n)))_i = (q\wedge(q\wedge d(y)))_i = (q\wedge d(y))_i.
\]
For $i\in\{1,\dots,k\}$ such that $d(y)_i=\infty$, we have
\[
 (q\wedge(q\wedge d(y)+d(\lambda)+(d(x)-m\wedge d(x)-n)))_i = q_i = (q\wedge d(y))_i.
\]
Hence $q\wedge d(y')=q\wedge d(y)$ and so $[y';q]=[y;q]$. Now,
consider $\mu := [y;(p,q)][y';(q,q+q\wedge d(y)+d(\lambda)+m)]$. We
have $r(\mu)=[y;p]$ and
\begin{eqnarray*}
 s(\mu) &=& [y';q+q\wedge d(y)+d(\lambda)+m] \\
 &=& [\sigma^{(q\wedge d(y)+d(\lambda))}(y');q+m] \\
 &=& [\sigma^{m\wedge d(x)+n}(x);q+m] \\
 &=& [x;m\wedge d(x)+n+q+m] \\
 &=& [x;(m,\infty)](m\wedge d(x)+n+q).
\end{eqnarray*}
Hence $\mu\in[y;p]\tLambda[x;(m,\infty)](m\wedge d(x)+n+q)$. As
$[y;p] \in \tLambda^0$ and $z = [x;(m,\infty)] \in \tLambda^\infty$
were arbitrary, it follows that $\tLambda$ is cofinal.

$(\Longleftarrow)$ Suppose $\tLambda$ is cofinal. Fix
$v\in\Lambda^0$ and $x\in\Lambda^{\leq\infty}$. Since
$\tLambda$ is cofinal we may fix $n\in\NN^k$ such that
\[
 v\tLambda[x;(0,\infty)](n)\neq\emptyset;
\]
say $\lambda\in v\tLambda[x;(0,\infty)](n)$. Then
\[
 r(\pi(\lambda)) = \pi(r(\lambda))= \pi(v) = v
\]
and
\[
 s(\pi(\lambda)) = \pi(s(\lambda))= \pi([x;n]) = [x;n\wedge d(x)]
\]
Hence $\pi(\lambda)\in v\Lambda[x;n\wedge d(x)]$. Once again, since
$v \in \Lambda^0$ and $x \in \Lambda^{\le\infty}$ were arbitrary, it
follows that $\Lambda$ is cofinal.
\end{proof}

\begin{prop} \label{prop:localperiodicity}
Let $(\Lambda,d)$ be a row-finite locally convex $k$-graph. Fix
$v \in \tLambda^0$ and $m \not= n \in \NN^k$. Then $\tLambda$
has local periodicity $m,n$ at $v$ if and only if $\Lambda$ has
local periodicity $m,n$ at $\pi(v)$. In particular, $\tLambda$
has no local periodicity if and only if $\Lambda$ has no local
periodicity.
\end{prop}

In order to prove this proposition, we require some preliminary
results.

\begin{lemma} \label{lemmasigmamap}
Let $(\Lambda,d)$ be a row-finite locally convex $k$-graph. Then for
any $[x;(n,\infty)]\in\tLambda^\infty$ and $m\in\NN^k$, we have
$\sigma^m([x;(n,\infty)]) = [x;(n+m,\infty)]$. If $m\leq d(x)$, we
also have $ \sigma^m([x;(n,\infty)]) = [\sigma^m(x);(n,\infty)]$.
\end{lemma}
\begin{proof}
For any $p\in\NN^k$, we have
\begin{eqnarray*}
 \sigma^m([x;(n,\infty)])(0,p) &=& [x;(n,\infty)](m,m+p) \\
 &=& [x;(n+m,n+m+p)] \\
 &=& [x;(n+m,\infty)](0,p).
\end{eqnarray*}
Since $\sigma^m([x;(n,\infty)])$ and $[x;(n+m,\infty)]$ agree
on every initial segment we conclude that they are equal. Now,
fix $m\leq d(x)$. We must show that
$\sigma^m([x;(n,\infty)])(0,p)=[\sigma^m(x);(n,\infty)](0,p)$~for
all $p\in\NN^k$, that is we must check that
$(x;(n+m,n+m+p))\approx(\sigma^m(x);(n,n+p))$. For (P1), fix
$p\in\NN^k$ and calculate
\begin{align*}
 \sigma^m(x)&(n\wedge d(\sigma^m(x)),(n+p)\wedge d(\sigma^m(x))) \\
 &= x(n\wedge d(\sigma^m(x))+m,(n+p)\wedge d(\sigma^m(x))+m) \\
 &= x((n+m)\wedge (d(\sigma^m(x))+m),(n+m+p)\wedge (d(\sigma^m(x))+m))
    \quad\text{by \ref{minimumdistributesoveraddition}} \\
 &= x((n+m)\wedge d(x),(n+m+p)\wedge d(x)).
\end{align*}
For (P2), we have
\begin{eqnarray*}
 n-n\wedge d(\sigma^m(x)) &=& (n+m)-(n\wedge d(\sigma^m(x))+m) \\
                          &=& (n+m)-(n+m)\wedge (d(\sigma^m(x))+m)
                            \quad\text{by \ref{minimumdistributesoveraddition}}\\
                          &=& (n+m)-(n+m)\wedge d(x).
\end{eqnarray*}
For (P3), we have $(n+m+p)-(n+m) = p = (n+p)-n$.
\end{proof}

\begin{lemma} \label{lemmalocallyperiodic1}
Let $\Lambda$ be a row-finite locally convex $k$-graph. Suppose that
$x \in \Lambda^{\le \infty}$ and $p,m,n \in \NN^k$ satisfy
$\sigma^m([x;(p,\infty)]) = \sigma^n([x;(p,\infty)])$. Then $d(x)_i =
0$ whenever $m_i\neq n_i$.
\end{lemma}
\begin{proof}
We prove contrapositive statement. Suppose $d(x)_i<\infty$ and fix
$l\in\NN^k$ such that $(p+m+l)_i>d(x)_i$ and $(p+n+l)_i>d(x)_i$. Then
\begin{align*}
 \sigma^m([x;&(p,\infty)])= \sigma^n([x;(p,\infty)]) \\
 &\Longrightarrow [x;(p+m,\infty)]=[x;(p+n,\infty)] \ \mbox{by Lemma \ref{lemmasigmamap}} \\
 &\Longrightarrow [x;(p+m,\infty)](0,l)=[x;(p+n,\infty)](0,l) \\
 &\Longrightarrow [x;(p+m,p+m+l)]=[x;(p+n,p+n+l)] \\ & \hspace{60mm} \mbox{by definition of } [x;(p,\infty)] \\
 &\Longrightarrow ((p+m+l)-(p+m+l)\wedge d(x))_i = ((p+n+l)-(p+n+l)\wedge d(x))_i  \\ & \hspace{60mm} \mbox{by (V2)}\\
 &\Longrightarrow ((p+m+l)-d(x))_i = ((p+n+l)-d(x))_i \\ & \hspace{60mm} \mbox{since } (p+m+l), (p+n+l) > d(x)_i \\
 &\Longrightarrow m_i=n_i.
\end{align*}
So $d(x)_i < \infty$ implies $m_i = n_i$ as claimed.
\end{proof}

\begin{lemma}\label{lem:periodic projection}
Let $\Lambda$ be a row-finite locally convex $k$-graph. Fix $y \in
\tLambda^\infty$ and $m,n \in \NN^k$. Then $\sigma^m(y) =
\sigma^n(y)$ if and only if: (a) $\sigma^{m \wedge d(\pi(y))}(\pi(y))
= \sigma^{n \wedge d(\pi(y))}(\pi(y))$; and~(b) $m - m \wedge
d(\pi(y)) = n - n \wedge d(\pi(y))$.
\end{lemma}
\begin{proof}
First suppose that (a)~and~(b) hold. By
Proposition~\ref{prop:infinitepaths}, $y =
[\pi(y);(p_y,\infty)]$ with $p_y \wedge d(\pi(y)) = 0$. Hence
\[
\sigma^{m \wedge d(\pi(y))}(y)
 = \sigma^{m \wedge d(\pi(y))}([\pi(y); (p_y,\infty)])
 = [\sigma^{m \wedge d(\pi(y))}(\pi(y)); (p_y, \infty)])
\]
by Lemma~\ref{lemmasigmamap}. Similarly, $\sigma^{n \wedge
d(\pi(y))}(y) = [\sigma^{n \wedge d(\pi(y))}(\pi(y)); (p_y,
\infty)])$, so~(a) implies that
\begin{equation}\label{eq:perp shifts equal}
\sigma^{m \wedge d(\pi(y))}(y) = \sigma^{n \wedge d(\pi(y))}(y).
\end{equation}
Using~(b) and~\eqref{eq:perp shifts equal}, we now calculate:
\[
\sigma^m(y)
 = \sigma^{m - m \wedge d(\pi(y))}(\sigma^{m \wedge d(\pi(y))}(y))
 = \sigma^{n - n \wedge d(\pi(y))}(\sigma^{n \wedge d(\pi(y))}(y))
 = \sigma^n(y).
\]

Now suppose that $\sigma^m(y) = \sigma^n(y)$. Using
Lemma~\ref{lemmasigmamap}, we see that
\[
\sigma^m(y)(0)
 = \sigma^m([\pi(y); (p_y,\infty)])(0)
 = [\pi(y); p_y + m]
\]
and similarly, $\sigma^n(y)(0) = [\pi(y); p_y + n]$. In particular,
condition~(V2) implies that
\[
(p_y + m) - ((p_y + m) \wedge d(\pi(y)))
 = (p_y + n) - ((p_y + n) \wedge d(\pi(y))).
\]
Since $p_y \wedge d(\pi(y)) = 0$, this establishes~(b). Note that by
Lemma~\ref{lemmalocallyperiodic1}, $d(\pi(y))_i = \infty$ whenever
$m_i \not= n_i$, and hence $d(\pi(y)) - m \wedge d(\pi(y)) =
d(\pi(y)) - n \wedge d(\pi(y))$. Thus $p \le d(\pi(y)) - m \wedge
d(\pi(y))$ if and only if $p \le d(\pi(y)) - n \wedge d(\pi(y))$, and
for such $p$ we may use the final statement of
Proposition~\ref{prop:infinitepaths} to calculate
\[
\sigma^{m \wedge d(\pi(y))}(\pi(y))(0,p)
 = \pi(y)(m \wedge d(\pi(y)), m \wedge d(\pi(y)) + p)
 = \pi(y(m,m + p)).
\]
Similarly, $\sigma^{n \wedge d(\pi(y))}(\pi(y))(0,p) = \pi(y(n,n +
p))$. We have $\sigma^m(y) = \sigma^n(y)$ by hypothesis, so $y(m,
m+p) = y(n, n+p)$, and hence $\pi(y(m,m + p)) = \pi(y(n,n + p))$.
Since $p \le d(\pi(y)) - m \wedge d(\pi(y)) = d(\pi(y)) - n \wedge
d(\pi(y))$ was arbitrary, this completes the proof.
\end{proof}

\begin{proof}[Proof of Proposition \ref{prop:localperiodicity}]
Lemma~\ref{lemmalocalperiodicity2} implies that for any $z \in
w\tLambda^\infty$, we have $p_z \wedge d(x) = 0$ and $[x; p_z] = w$
for all $x \in \pi(w)\Lambda^{\le \infty}$. It follows that for each
$x \in \pi(w)\Lambda^{\le \infty}$ we have $[x;(p_z, \infty)]$ in
$w\tLambda^\infty$, and $x = \pi([x;(p_z, \infty)])$. That is,
$\pi(w)\Lambda^{\le \infty} = \{\pi(z) : z \in w\tLambda^\infty\}$.
The result now follows from Lemma~\ref{lem:periodic projection}.
\end{proof}

\begin{proof}[Proof of Theorem \ref{theorem:simplicity}]
Propositions \ref{prop:cofinal} and \ref{prop:localperiodicity}
imply that $\Lambda$ is cofinal and has no local periodicity if
and only if $\tLambda$ is cofinal and has no local periodicity.
Since $\tLambda$ is row-finite and has no sources by
\cite[Theorem 3.23]{Fa}, Theorem 3.3 of \cite{RoSi} implies
that $C^*(\tLambda)$ is simple if and only if $\Lambda$ is
cofinal and has no local periodicity. Theorem 3.29 of \cite{Fa}
implies that $C^*(\Lambda)$ is a full corner of
$C^*(\tLambda)$, so by \cite[Theorem 3.19]{RaWi} $C^*(\Lambda)$
and $C^*(\tLambda)$ are Morita equivalent. In particular,
$C^*(\Lambda)$ is simple if and only if $C^*(\tLambda)$ is
simple. Bringing these implications together,
\begin{align*}
 C^*(\Lambda) \mbox{ is simple } & \iff C^*(\tLambda) \mbox{ is simple} \\
 & \iff \tLambda \mbox{ is cofinal and has no local periodicity} \\
 & \iff \Lambda \mbox{ is cofinal and has no local periodicity}
\end{align*}
as required.
\end{proof}

Before concluding this section, we pause to discuss, briefly,
the local periodicity condition presented in
Definition~\ref{dfn:nlp}. This definition is not, perhaps, the
obvious extrapolation of the condition given in \cite{RoSi} to
the locally convex setting (though certainly
Proposition~\ref{prop:localperiodicity} indicates that it is
the right one). The more obvious definition would be to say
that $\Lambda$ has local periodicity $p,q$ at $w$ if
\begin{equation}\label{eq:oldnlp}
\parbox{0.75\textwidth}{for every $x \in w\Lambda^{\le\infty}$ we have $p,q \le d(x)$
and $\sigma^p(x) = \sigma^q(x)$.}
\end{equation}
The two are not equivalent: if $\Lambda$ is the 2-graph whose
skeleton appears on the left of Figure~\ref{fig:ex}, then one can
check that the skeleton of $\tLambda$ is that which appears on the
right of Figure~\ref{fig:ex}. One can also check that $\tLambda$ has
local periodicity $(1,2), (0,2)$ at $v = \pi(v)$, but there do not
exist $p,q \in \NN^2$ satisfying~\eqref{eq:oldnlp} with $w = v$.
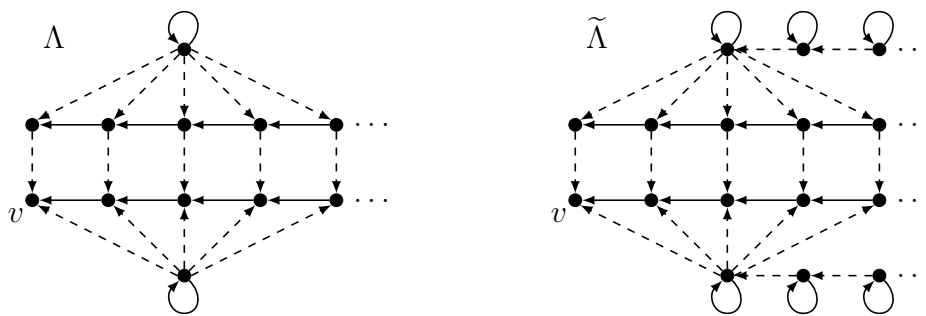
\begin{figure}[ht]
\[
\hbox to 0.35\textwidth{\hfil
\begin{tikzpicture}[scale=0.5]
    \node[anchor=south west] at (0,3.75) {$\Lambda$};
    \node[inner sep=0.5pt, circle, fill=black] (v) at (0,0) [draw] {\phantom{.}};
    \node[inner sep=1pt, anchor=north east] at (v.south west) {$v$};
    \node[inner sep=0.5pt, circle, fill=black] (1) at (2,0) [draw] {\phantom{.}};
    \node[inner sep=0.5pt, circle, fill=black] (2) at (4,0) [draw] {\phantom{.}};
    \node[inner sep=0.5pt, circle, fill=black] (3) at (6,0) [draw] {\phantom{.}};
    \node[inner sep=0.5pt, circle, fill=black] (4) at (8,0) [draw] {\phantom{.}};
    \node[inner sep=0.5em] (dots) at (9,0) {\dots};
    \node[inner sep=0.5pt, circle, fill=black] (v+) at (0,2) [draw] {\phantom{.}};
    \node[inner sep=0.5pt, circle, fill=black] (1+) at (2,2) [draw] {\phantom{.}};
    \node[inner sep=0.5pt, circle, fill=black] (2+) at (4,2) [draw] {\phantom{.}};
    \node[inner sep=0.5pt, circle, fill=black] (3+) at (6,2) [draw] {\phantom{.}};
    \node[inner sep=0.5pt, circle, fill=black] (4+) at (8,2) [draw] {\phantom{.}};
    \node[inner sep=0.5em] (dots+) at (9,2) {\dots};
    \node[inner sep=0.5pt, circle, fill=black] (w) at (4,4) [draw] {\phantom{.}};
    \node[inner sep=0.5pt, circle, fill=black] (x) at (4,-2) [draw] {\phantom{.}};
%
    \draw[style=semithick,-latex] (1.west)--(v.east);
    \draw[style=semithick,-latex] (2.west)--(1.east);
    \draw[style=semithick,-latex] (3.west)--(2.east);
    \draw[style=semithick,-latex] (4.west)--(3.east);
    \draw[style=semithick,-latex] (dots.west)--(4.east);
    \draw[style=semithick,-latex] (1+.west)--(v+.east);
    \draw[style=semithick,-latex] (2+.west)--(1+.east);
    \draw[style=semithick,-latex] (3+.west)--(2+.east);
    \draw[style=semithick,-latex] (4+.west)--(3+.east);
    \draw[style=semithick,-latex] (dots+.west)--(4+.east);
    \draw[style=semithick,dashed,-latex] (v+.south)--(v.north);
    \draw[style=semithick,dashed,-latex] (1+.south)--(1.north);
    \draw[style=semithick,dashed,-latex] (2+.south)--(2.north);
    \draw[style=semithick,dashed,-latex] (3+.south)--(3.north);
    \draw[style=semithick,dashed,-latex] (4+.south)--(4.north);
    \draw[style=semithick,-latex] (w.north east) .. controls (4.5,4.5) and (4.4,5) .. (4,5)
                                                 .. controls (3.6,5) and (3.5,4.5) .. (w.north west);
    \draw[style=semithick,-latex] (x.south east) .. controls (4.5,-2.5) and (4.4,-3) .. (4,-3)
                                                 .. controls (3.6,-3) and (3.5,-2.5) .. (x.south west);
    \draw[style=semithick,dashed,-latex] (w.west)--(v+.north east);
    \draw[style=semithick,dashed,-latex] (w.south west)--(1+.north east);
    \draw[style=semithick,dashed,-latex] (w.south)--(2+.north);
    \draw[style=semithick,dashed,-latex] (w.south east)--(3+.north west);
    \draw[style=semithick,dashed,-latex] (w.east)--(4+.north west);
    \draw[style=semithick,dashed,-latex] (x.west)--(v.south east);
    \draw[style=semithick,dashed,-latex] (x.north west)--(1.south east);
    \draw[style=semithick,dashed,-latex] (x.north)--(2.south);
    \draw[style=semithick,dashed,-latex] (x.north east)--(3.south west);
    \draw[style=semithick,dashed,-latex] (x.east)--(4.south west);
\end{tikzpicture}\hfil
}
\hskip0.1\textwidth
\hbox to 0.35\textwidth{\hfil
\begin{tikzpicture}[scale=0.5]
    \node[anchor=south west] at (0,3.75) {$\tLambda$};
    \node[inner sep=0.5pt, circle, fill=black] (v) at (0,0) [draw] {\phantom{.}};
    \node[inner sep=1pt, anchor=north east] at (v.south west) {$v$};
    \node[inner sep=0.5pt, circle, fill=black] (1) at (2,0) [draw] {\phantom{.}};
    \node[inner sep=0.5pt, circle, fill=black] (2) at (4,0) [draw] {\phantom{.}};
    \node[inner sep=0.5pt, circle, fill=black] (3) at (6,0) [draw] {\phantom{.}};
    \node[inner sep=0.5pt, circle, fill=black] (4) at (8,0) [draw] {\phantom{.}};
    \node[inner sep=0.5em] (dots) at (9,0) {\dots};
    \node[inner sep=0.5pt, circle, fill=black] (v+) at (0,2) [draw] {\phantom{.}};
    \node[inner sep=0.5pt, circle, fill=black] (1+) at (2,2) [draw] {\phantom{.}};
    \node[inner sep=0.5pt, circle, fill=black] (2+) at (4,2) [draw] {\phantom{.}};
    \node[inner sep=0.5pt, circle, fill=black] (3+) at (6,2) [draw] {\phantom{.}};
    \node[inner sep=0.5pt, circle, fill=black] (4+) at (8,2) [draw] {\phantom{.}};
    \node[inner sep=0.5em] (dots+) at (9,2) {\dots};
    \node[inner sep=0.5pt, circle, fill=black] (w) at (4,4) [draw] {\phantom{.}};
    \node[inner sep=0.5pt, circle, fill=black] (x) at (4,-2) [draw] {\phantom{.}};
%
    \draw[style=semithick,-latex] (1.west)--(v.east);
    \draw[style=semithick,-latex] (2.west)--(1.east);
    \draw[style=semithick,-latex] (3.west)--(2.east);
    \draw[style=semithick,-latex] (4.west)--(3.east);
    \draw[style=semithick,-latex] (dots.west)--(4.east);
    \draw[style=semithick,-latex] (1+.west)--(v+.east);
    \draw[style=semithick,-latex] (2+.west)--(1+.east);
    \draw[style=semithick,-latex] (3+.west)--(2+.east);
    \draw[style=semithick,-latex] (4+.west)--(3+.east);
    \draw[style=semithick,-latex] (dots+.west)--(4+.east);
    \draw[style=semithick,dashed,-latex] (v+.south)--(v.north);
    \draw[style=semithick,dashed,-latex] (1+.south)--(1.north);
    \draw[style=semithick,dashed,-latex] (2+.south)--(2.north);
    \draw[style=semithick,dashed,-latex] (3+.south)--(3.north);
    \draw[style=semithick,dashed,-latex] (4+.south)--(4.north);
    \draw[style=semithick,-latex] (w.north east) .. controls (4.5,4.5) and (4.4,5) .. (4,5)
                                                 .. controls (3.6,5) and (3.5,4.5) .. (w.north west);
    \draw[style=semithick,-latex] (x.south east) .. controls (4.5,-2.5) and (4.4,-3) .. (4,-3)
                                                 .. controls (3.6,-3) and (3.5,-2.5) .. (x.south west);
    \draw[style=semithick,dashed,-latex] (w.west)--(v+.north east);
    \draw[style=semithick,dashed,-latex] (w.south west)--(1+.north east);
    \draw[style=semithick,dashed,-latex] (w.south)--(2+.north);
    \draw[style=semithick,dashed,-latex] (w.south east)--(3+.north west);
    \draw[style=semithick,dashed,-latex] (w.east)--(4+.north west);
    \draw[style=semithick,dashed,-latex] (x.west)--(v.south east);
    \draw[style=semithick,dashed,-latex] (x.north west)--(1.south east);
    \draw[style=semithick,dashed,-latex] (x.north)--(2.south);
    \draw[style=semithick,dashed,-latex] (x.north east)--(3.south west);
    \draw[style=semithick,dashed,-latex] (x.east)--(4.south west);
%
%
    \node[inner sep=0.5pt, circle, fill=black] (3-) at (6,-2) [draw] {\phantom{.}};
    \node[inner sep=0.5pt, circle, fill=black] (4-) at (8,-2) [draw] {\phantom{.}};
    \node[inner sep=0.5em] (dots-) at (9,-2) {\dots};
    \draw[style=semithick,dashed,-latex] (3-.west)--(x.east);
    \draw[style=semithick,dashed,-latex] (4-.west)--(3-.east);
    \draw[style=semithick,dashed,-latex] (dots-.west)--(4-.east);
    \draw[style=semithick,-latex] (3-.south east) .. controls (6.5,-2.5) and (6.4,-3) .. (6,-3)
                                                 .. controls (5.6,-3) and (5.5,-2.5) .. (3-.south west);
    \draw[style=semithick,-latex] (4-.south east) .. controls (8.5,-2.5) and (8.4,-3) .. (8,-3)
                                                 .. controls (7.6,-3) and (7.5,-2.5) .. (4-.south west);
    \node[inner sep=0.5pt, circle, fill=black] (3++) at (6,4) [draw] {\phantom{.}};
    \node[inner sep=0.5pt, circle, fill=black] (4++) at (8,4) [draw] {\phantom{.}};
    \node[inner sep=0.5em] (dots++) at (9,4) {\dots};
    \draw[style=semithick,dashed,-latex] (3++.west)--(w.east);
    \draw[style=semithick,dashed,-latex] (4++.west)--(3++.east);
    \draw[style=semithick,dashed,-latex] (dots++.west)--(4++.east);
    \draw[style=semithick,-latex] (3++.north east) .. controls (6.5,4.5) and (6.4,5) .. (6,5)
                                                 .. controls (5.6,5) and (5.5,4.5) .. (3++.north west);
    \draw[style=semithick,-latex] (4++.north east) .. controls (8.5,4.5) and (8.4,5) .. (8,5)
                                                 .. controls (7.6,5) and (7.5,4.5) .. (4++.north west);
\end{tikzpicture}\hfil
}
\]
\caption{Local periodicity $p,q$ at $w$ is not equivalent to Equation~\eqref{eq:oldnlp}}\label{fig:ex}
\end{figure}
That having been said, the key notion for the purposes of
characterising simplicity is that of \emph{no} local
periodicity, and the following lemma shows that
Definition~\ref{dfn:nlp} and Equation~\eqref{eq:oldnlp}
correspond to the same notion of no local periodicity.

\begin{lemma}
Let $\Lambda$ be a row-finite locally convex $k$-graph. Then
there exist $v \in \Lambda^0$ and $m \not= n \in \NN^k$ such
that $\Lambda$ has local periodicity at $v$ if and only if
there exist $w \in \Lambda^0$ and $p,q \in \NN^k$
satisfying~\eqref{eq:oldnlp}.
\end{lemma}
\begin{proof}
Suppose first that $w,p,q$ satisfy~\eqref{eq:oldnlp}. Then in
particular, for $x \in w\Lambda^{\le\infty}$, we have $p - p \wedge
d(x) = 0 = q - q \wedge d(x)$, and
\[
\sigma^{p \wedge d(x)}(x)
 = \sigma^p(x)
 = \sigma^q(x)
 = \sigma^{q \wedge d(x)}(x).
\]
Hence for $m = p$, $n = q$ and $v = w$, $\Lambda$ has local
periodicity $p,q$ at $w$.

Now suppose that $\Lambda$ has local periodicity $m,n$ at $v$.
Then Proposition~\ref{prop:localperiodicity} implies that
$\tLambda$ has local periodicity $m,n$ at $v$. Let $p := m - m
\wedge n$ and $q := n - m \wedge n$, fix $\lambda \in
v\tLambda^{m \wedge n}$, and let $u := s(\lambda)$. For $x \in
u\tLambda^\infty$,
\[
\sigma^p(x) = \sigma^p(\sigma^{m \wedge n}(\lambda x)) = \sigma^m(\lambda x),
\]
and likewise $\sigma^q(x) = \sigma^n(\lambda x)$. Since
$\lambda x \in v\tLambda^\infty$, it follows that $\sigma^p(x)
= \sigma^q(x)$. So $\tLambda$ has local periodicity $p,q$ at
$u$, and another application of
Proposition~\ref{prop:localperiodicity} then shows that
$\Lambda$ has local periodicity $p,q$ at $\pi(u)$. Since $p
\wedge q = 0$, Lemma~\ref{lemmalocallyperiodic1} implies that
$p,q \le d(x)$ for all $x \in \pi(u)\Lambda^{\le\infty}$. Hence
$w := \pi(u)$, $p = m - m \wedge n$ and $q = n - m \wedge n$
satisfy~\eqref{eq:oldnlp}.
\end{proof}

\section{Ideal Structure} \label{sec:idealstructure}

In this section we describe the relationship between cofinality
and local periodicity of $\Lambda$ and the ideal structure of
$C^*(\Lambda)$. In order to state the results we need some
background. For details of the following, see \cite{KuPa}.

For $z\in\TT^k$ and $n\in\NN^k$, we use the multi-index notation
$\textstyle z^n := \prod_{i=1}^k z_i^{n_i}\in\TT$. There is a
strongly continuous action $\gamma$ of $\TT^k$ on $C^*(\Lambda)$
satisfying $\gamma_z(s_\lambda)=z^{d(\lambda)}s_\lambda$ for all
$\lambda\in\Lambda$. The fixed point algebra $C^*(\Lambda)^\gamma$ is
called the \emph{core} of $C^*(\Lambda)$ and is equal to
$\clsp\{s_\mu s_\nu^*:d(\mu)=d(\nu)\}$.

The following lemma will be useful in proving both of our main
results in this section. In the statement of the lemma and what
follows, $\Mm(A)$ denotes the multiplier algebra of a
$C^*$-algebra $A$.

\begin{lemma} \label{lemma:corner}
Let $\Lambda$ be a row-finite locally convex $k$-graph. Let
$\widetilde{I}$ be an ideal of $C^*(\tLambda)$ and let $P := \sum_{v
\in \Lambda^0} s_v \in\mathcal{M}(C^*(\tLambda))$. Then
\begin{enumerate}
 \item for $v \in \tLambda^0$, we have $s_v \in
     \widetilde{I}$ if and only if $s_{\pi(v)} \in P
     \widetilde{I} P$; and
\item $P \widetilde{I} P \cap C^*(\Lambda)^\gamma
    \not= \{0\}$ if and only if $\widetilde{I} \cap
    C^*(\tLambda)^\gamma \not= \{0\}$.
\end{enumerate}
\end{lemma}

\begin{proof}
For~(1), fix $v \in \tLambda^0$, and use
Lemma~\ref{lemma:pathnotation} to write $v = [x; p]$ where $x \in
\pi(v)\Lambda^{\le\infty}$ and $p \wedge d(x) = 0$. Let $\lambda :=
[x; (0, p)] \in \pi(v)\tLambda v$. Lemma~\ref{lemmalocalperiodicity2}
implies that $\pi(v)\tLambda^p = \{\lambda\}$. Hence the
Cuntz-Krieger relations show that
\[
s_v = s^*_\lambda s_{\pi(v)} s_\lambda
\quad\text{ and }\quad
s_{\pi(v)} = s_\lambda s_{v} s^*_\lambda.
\]
Since $P s_{\pi(v)} P = s_{\pi(v)}$, this proves~(a).

For~(2), the ``only if'' implication is trivial, so it suffices to
establish the ``if" direction. Suppose that $\widetilde{I} \cap
C^*(\tLambda)^\gamma \not= \{0\}$. The argument of
(ii)$\implies$(iii) in \cite[Proposition~3.4]{RoSi} shows that there
exists $w \in \tLambda^0$ such that $s_w \in \widetilde{I}$.
Hence~(a) implies that there exists $v \in \Lambda^0$ such that $s_v
\in P\widetilde{I}P$, and since the vertex projections are fixed by
$\gamma$, we then have  $s_v \in P\widetilde{I}P \cap
C^*(\Lambda)^\gamma$.
\end{proof}

\begin{prop} \label{prop:cofinalequiv}
Let $\Lambda$ be a row-finite locally convex $k$-graph. The following are equivalent.
\begin{itemize}
\item[(1)] $\Lambda$ is cofinal
\item[(2)] If $I$ is an ideal of $C^*(\Lambda)$ and $s_v\in
    I$ for some $v\in\Lambda^0$, then $I=C^*(\Lambda)$.
\item[(3)] If $I$ is an ideal of $C^*(\Lambda)$ and $I\cap
    C^*(\Lambda)^\gamma\neq\{0\}$, then $I=C^*(\Lambda)$.
\end{itemize}
\end{prop}
\begin{proof}
That these three statements are equivalent for
$\widetilde{\Lambda}$ follows from \cite[Proposition~3.5]{RoSi}
because $\widetilde{\Lambda}$ has no sources.
Proposition~\ref{prop:cofinal} shows that $\widetilde{\Lambda}$
is cofinal if and only if $\Lambda$ is cofinal. Recall that
$\sum_{v \in \Lambda^0} s_v$ converges to a full projection $P
\in \Mm C^*(\widetilde{\Lambda})$ and that $P C^*(\tLambda) P$
is canonically isomorphic to $C^*(\Lambda)$. Hence the map
$\widetilde{I} \mapsto P \widetilde{I} P$ is a bijection
between ideals of $C^*(\tLambda)$ and ideals of $\Lambda$. The
proposition then follows from Lemma \ref{lemma:corner}.
\end{proof}

In order to state the next result we need more background. As
in \cite[Theorem 3.15]{RaSiYe} let
$\mathcal{H}:=\ell^2(\Lambda^{\leq\infty})$ with standard basis
denoted $\{u_x:x\in\Lambda^{\leq\infty}\}$. For each
$\lambda\in\Lambda$ define $S_\lambda\in
\mathcal{B}(\mathcal{H})$ by
\[
 S_\lambda u_x = \left\{
 \begin{array}{ll} u_{\lambda x} & \mbox{if } s(\lambda)=r(x) \\
  0 & \mbox{otherwise.}
 \end{array} \right.
\]
Then $\{S_\lambda:\lambda\in\Lambda\}\subset \mathcal{H}$ is a
Cuntz-Krieger $\Lambda$-family. By the universal property of
$C^*(\Lambda)$ there exists a homomorphism $\pi_S:C^*(\Lambda)\to
\mathcal{B}(\mathcal{H})$ such that $\pi_S(s_\lambda)=S_\lambda$ for
all $\lambda\in\Lambda$. We call $\pi_S$ the \emph{boundary-path
representation}.

\begin{prop} \label{prop:localperiodicityequiv}
Let $\Lambda$ be a row-finite locally convex $k$-graph. The following are equivalent:
\begin{itemize}
 \item[(1)] $\Lambda$ has no local periodicity.
 \item[(2)] Every non-zero ideal of $C^*(\Lambda)$ contains a vertex projection.
 \item[(3)] The boundary-path representation $\pi_S$ is faithful
\end{itemize}
\end{prop}

In order to prove this proposition we need a technical lemma.

\begin{lemma} \label{lemma:technical1}
Let $\Lambda$ be a row-finite locally convex $k$-graph and suppose
that $\Lambda$ has local periodicity $m,n$ at $v$. Fix $x \in
v\Lambda^{\le \infty}$, let $\mu := x(0, m \wedge d(x))$, let $\alpha
:= x(m \wedge d(x), (m \vee n) \wedge d(x))$, and let $\nu := x(0, n
\wedge d(x))$. Then $d(\mu) \not = d(\nu)$, and $\mu\alpha
y=\nu\alpha y$ for all $y\in s(\alpha)\Lambda^{\leq\infty}$.
\end{lemma}
\begin{proof}
Since $x\in v\Lambda^{\leq\infty}$ we have $\sigma^{m \wedge
d(x)}(x)=\sigma^{n\wedge d(x)}(x)$ and $m - m \wedge d(x) = n - n
\wedge d(x)$. Since $m \not=n$ and $m - m \wedge d(x) = n - n \wedge
d(x)$, we immediately have $d(\mu) \not = d(\nu)$.

Fix $y\in s(\alpha)\Lambda^{\leq\infty}$; we must show\ that
$\mu\alpha y=\nu\alpha y$. Let $z:=\mu\alpha y$. Then $z \in
v\Lambda^{\le\infty}$, and hence $\sigma^{m \wedge d(z)}(z) =
\sigma^{n \wedge d(z)}(z)$.

We claim that $m \wedge d(z) = m \wedge d(x)$. To see this, fix $1
\le i \le k$. If $m_i \le d(x)_i$, then $d(z)_i \ge d(\mu)_i = m_i$,
so we also have $m_i \le d(z_i)$, and $(m \wedge d(x))_i = m_i = (m
\wedge d(z))_i$. If $m_i > d(x)_i$,  then $x \in \Lambda^{\le
\infty}$ forces $s(\mu)\Lambda^{e_i} = \emptyset$. Hence $d(z)_i =
d(\mu)_i =  d(x)_i$, and $(m \wedge d(x))_i = d(x)_i = d(z)_i = (m
\wedge d(z))_i$. This establishes the claim. A similar argument shows
that $n \wedge d(z) = n \wedge d(x)$.

We now have
\begin{equation}\label{eq:alphay=sigma(z)}
\alpha y
 = \sigma^{d(\mu)}(z)
 = \sigma^{m \wedge d(x)}(z)
 = \sigma^{m \wedge d(z)}(z)
 = \sigma^{n \wedge d(z)}(z)
 = \sigma^{n \wedge d(x)}(z).
\end{equation}
We now calculate coordinate-wise to see that $(m \vee n) \wedge d(x)
= (m \wedge d(x)) \vee (n \wedge d(x))$. In particular, $d(\mu\alpha)
\ge d(\nu)$, so $\nu=(\mu\alpha)(0,d(\nu))$. By definition of $\nu$,
we have $z=\nu\sigma^{n \wedge d(x)}(z)$.  So $\mu\alpha y = z =
\nu\sigma^{n \wedge d(x)}(z)$, and this is equal to $\nu\alpha y$
by~\eqref{eq:alphay=sigma(z)}.
\end{proof}

\begin{proof}[Proof of Proposition \ref{prop:localperiodicityequiv}]
($(1)\Longrightarrow(2)$) Suppose $\Lambda$ has no local periodicity
and let $I$ be an ideal in $C^*(\Lambda)$. Then $I=P\widetilde{I}P$
for some ideal $\widetilde{I}$ of $C^*(\tLambda)$. By
Proposition~\ref{prop:localperiodicity} $\tLambda$ has no local
periodicity so \cite[Proposition 3.6]{RoSi} implies that
$\widetilde{I}$ contains a vertex projection $s_{[x;m]}$. Lemma
\ref{lemma:corner}(1) then implies that $I=P\widetilde{I}P$ also
contains a vertex projection.

($(2)\Longrightarrow(3)$) For $v\in\Lambda^0$, $\pi_S(s_v)=S_v$
is the projection onto $\clsp\{u_x:x\in\Lambda^{\leq\infty}\}$
and so is non-zero. So $\ker(\pi_S)$ contains no vertex
projection and is trivial by (2).

($(3)\Longrightarrow(1)$) The proof of this implication runs almost
exactly the same as in the proof of \cite[Proposition 3.6]{RoSi}, but
we substitute Lemma \ref{lemma:technical1} for \cite[Lemma
3.4]{RoSi}. The broad strategy is as follows: we argue by
contrapositive, supposing $\Lambda$ has local periodicity $m,n$ at
$v$. Lemma~\ref{lemma:technical1} implies that there exist distinct
elements $m', n'$ of $\NN^k$ and paths $\mu\in v\Lambda^{m'}, \nu\in
v\Lambda^{n'} s(\mu)$ and $\alpha\in s(\mu)\Lambda$ such that
$\mu\alpha y =\nu\alpha y$ for all $y\in
s(\alpha)\Lambda^{\leq\infty}$. Let
\[
 a:=s_{\mu\alpha}s_{\mu\alpha}^*-s_{\nu\alpha}s_{\nu\alpha}^*
\]
We use the gauge action to prove that $a\neq 0$, and show
directly that $a\in \ker(\pi_S)$, so the boundary path
representation is not faithful.
\end{proof}

Before we state the next result we need to recall some
terminology from \cite[Section 5]{RaSiYe}.  We say a subset
$H\subseteq \Lambda^0$ is \emph{hereditary} if $r(\lambda)\in
H$ implies $s(\lambda)\in H$ for all $\lambda\in\Lambda$. We
say $H$ is \emph{saturated} if
\[
 \{s(\lambda):\lambda\in v\Lambda^{\leq e_i}\}\subseteq H
    \text{ for some } i\in\{1,\dots,k\} \text{ implies } v\in H.
\]
If $H\subset\Lambda^0$ is saturated and hereditary, then
$\Lambda\setminus\Lambda H$ is a sub-$k$-graph of $\Lambda$.

\begin{prop}
Let $\Lambda$ be a row-finite locally convex $k$-graph. Then
the following are equivalent.
\begin{itemize}
\item[(1)] Every ideal of $C^*(\Lambda)$ is
    gauge-invariant.
\item[(2)] For every saturated hereditary subset $H\subset
    \Lambda^0$, $\Lambda\setminus\Lambda H$ has no local
    periodicity.
\end{itemize}
\end{prop}
\begin{proof}
($(2)\Longrightarrow(1)$) Suppose $\Lambda\setminus\Lambda H$ has no
local periodicity for each saturated hereditary subset $H$ of
$\Lambda^0$. Fix a non-zero ideal $I$ of $C^*(\Lambda)$ and let
$H:=\{v\in\Lambda^0:s_v\in I\}$. Theorem 5.2 of \cite{RaSiYe}
guarantees that this set is saturated and hereditary. Let
$\{t_\lambda:\lambda\in\Lambda\setminus\Lambda H\}$ be the universal
generating Cuntz-Krieger family in $C^*(\Lambda\setminus\Lambda H)$.
Theorem~5.2 of~\cite{RaSiYe} also implies that the ideal $I_H$
generated by $\{s_v:v\in H\}$ is a gauge invariant ideal of
$C^*(\Lambda)$, and that there is an isomorphism
$\varphi:C^*(\Lambda\setminus\Lambda H)\to C^*(\Lambda)/ I_H$
satisfying $\varphi(t_\lambda)=s_\lambda+I_H$ for all
$\lambda\in\Lambda\setminus\Lambda H$. So it suffices to show that
$I=I_H$. Since $I_H\subset I$, the quotient map $q_I:C^*(\Lambda)\to
C^*(\Lambda)/I$ defined by
\[
 q_I(s_\lambda)=s_\lambda+I
\]
descends to a $C^*$-homomorphism
$\widetilde{q_I}:C^*(\Lambda)/I_H\to C^*(\Lambda)/I$ satisfying
\[
 \widetilde{q_I}(s_\lambda+I_{H_I})=s_\lambda+I
\]
for all $\lambda\in\Lambda$. Consider the composition
\[
 \widetilde{q_I}\circ \varphi:C^*(\Lambda\setminus\Lambda H)\to C^*(\Lambda)/I.
\]
We claim $\widetilde{q_I}\circ\varphi(t_v)\neq0$ for all
$v\in\Lambda^0\setminus H=(\Lambda\setminus\Lambda H)^0$.

To see this, fix $v\in(\Lambda\setminus\Lambda H)^0$. Then
\[
 v\not\in H
 \Longrightarrow s_v\not\in I
 \Longrightarrow q_I(s_v)\neq 0
 \Longrightarrow \widetilde{q_I}(s_v+I_{H_I})\neq 0
 \Longrightarrow \widetilde{q_I}\circ\varphi(t_v)\neq 0.
\]
Hence $\ker(\widetilde{q_I}\circ\varphi)$ contains no vertex
projections and Proposition~\ref{prop:localperiodicityequiv}
implies that $\ker(\widetilde{q_I}\circ\varphi)=\{0\}$. So
$\widetilde{q_I}\circ\varphi$ is an isomorphism, and in
particular $\widetilde{q_I}:C^*(\Lambda)/I_{H_I}\to
C^*(\Lambda)/I$ is injective. This forces $I_H=I$ as required.

($(1)\Longrightarrow(2)$) The proof of this implication is
almost identical to the proof of (1)$\Longrightarrow$(2) in
\cite[Proposition 3.7]{RoSi}. We argue by contrapositive.
Suppose there are a saturated hereditary subset
$H\subset\Lambda$, a vertex $v\in\Lambda^0\setminus H$ and
elements $m\neq n\in\NN^k$ such that $\Lambda\setminus\Lambda
H$ has local periodicity $m,n$ at $v$. Let
$\{t_\lambda:\lambda\in\Lambda\setminus\Lambda H\}$ be the
universal generating Cuntz-Krieger  family for
$C^*(\Lambda\setminus\Lambda H)$. Theorem 5.2 of \cite{RaSiYe}
guarantees there exists an isomorphism
$\varphi:C^*(\Lambda\setminus\Lambda H)\to C^*(\Lambda)/I_H$
satisfying $\varphi(t_\lambda)=s_\lambda+I_H$. Let $q_{I_H}$
denote the quotient map from $C^*(\Lambda)$ to
$C^*(\Lambda)/I_H$.

As in the proof of Proposition \ref{prop:localperiodicityequiv}, we
construct $\mu, \nu, \alpha \in \Lambda \setminus \Lambda H$ such
that $a = t_{\mu\alpha}t_{\mu\alpha}^*-t_{\nu\alpha}t_{\nu\alpha}^*
\in C^*(\Lambda\setminus\Lambda H)\setminus\{0\}$ such that
$\pi_T(a)=0$ where $\pi_T$ is the boundary path representation of
$C^*(\Lambda \setminus \Lambda H)$. Let
$b:=s_{\mu\alpha}s_{\mu\alpha}^*-s_{\nu\alpha}s_{\nu\alpha}^*\in
C^*(\Lambda)$. Then $\varphi\circ q_{I_H}(b)=a\neq0$ but
$\pi_T\circ\varphi\circ q_{I_H}(b)=\pi_T(a)=0$. Since $\varphi$ is an
isomorphism, the kernel of $\varphi\circ q_{I_H}$ is $I_H$. Since the
kernel of $\pi_T$ contains no vertices of $\Lambda\setminus\Lambda
H$, the ideal $J=\ker(\pi_T\circ\varphi\circ q_{I_H})$ also satisfies
$J_H=H$. We have $J\neq I_H$ because $b\in J\setminus I_H$. Theorem
5.2 of \cite{RaSiYe} implies that $H\mapsto I_H$ is an isomorphism of
the saturated hereditary subsets of $\Lambda$ and gauge-invariant
ideals of $C^*(\Lambda)$, and the implication $(2)\Longrightarrow
(1)$ shows that this has inverse $I\mapsto I_H$. Since $J\neq I =
I_{H_J}$ it follows that $J$ is a non-trivial ideal of $C^*(\Lambda)$
that is not gauge invariant.
\end{proof}

\end{document}